\documentclass[11pt,a4paper]{amsart}
\usepackage{amssymb}
\usepackage{enumerate}
\usepackage{graphicx}
\usepackage{psfrag}

\usepackage{amsmath}
\usepackage{amsfonts}
\usepackage{amsthm}
\usepackage{color}

 \newtheorem{teo}{Theorem}
 \newtheorem{lema}{Lemma}
 
 \newtheorem{nota}{Remark}

 \newtheorem{prop}{Proposition}
 \newtheorem{ex}{Example}

\DeclareMathOperator{\sign}{sign}

\begin{document}

\title[Limit cycles for a class of quintic $\mathbb{Z}_6-$equivariant equations]
{Limit cycles for a class of quintic $\mathbb{Z}_6-$equivariant systems without infinite critical points}
\author{M.J. \'Alvarez}
\address{M.J. \'Alvarez, Departament de Matem{\`a}tiques i Inform{\`a}tica, Universitat de les Illes Ba\-lears, 07122, Palma de Mallorca, Spain}
\email{chus.alvarez@uib.es}
\author{I.S. Labouriau}
\address{I.S. Labouriau, Centro de Matem\'atica da Universidade do Porto.\\ Rua do Campo Alegre 687, 4169-007 Porto, Portugal}
\email{islabour@fc.up.pt}
\author{A.C. Murza}
\address{A.C. Murza, Centro de Matem\'atica da Universidade do Porto.\\ Rua do Campo Alegre 687, 4169-007 Porto, Portugal}
\email{adrian.murza@fc.up.pt}

\thanks{M.J.A. was partially supported by grant MTM2008-03437. I.S.L. and A.C.M. were partially supported by
the European Regional Development Fund through the programme COMPETE and  through the Fun\-da\-\c c\~ao para
a Ci\^encia e a Tecnologia (FCT) under
the project PEst-C/MAT/UI0144/2013.
A.C.M. was also supported by the grant SFRH/ BD/ 64374/ 2009 of FCT
}

\begin{abstract}
We analyze the dynamics of a class of $\mathbb{Z}_6-$equivariant systems of the form $\dot{z}=pz^2\bar{z}+sz^3\bar{z}^2-\bar{z}^{5},$ where $z$ is
complex, the time $t$ is real, while $p$ and $s$
 are complex parameters. This study is the natural continuation of a previous work
(M.J. \'Alvarez, A. Gasull, R. Prohens,  Proc. Am. Math. Soc. \textbf{136}, (2008), 1035--1043)
on the normal form of $\mathbb{Z}_4-$equivariant systems. Our study uses the reduction of the equation to an Abel one,
and provide criteria for proving in some cases uniqueness and hyperbolicity of the limit cycle surrounding either 1, 7 or 13 critical points,
the origin being always one of these points.
\end{abstract}

\maketitle

\bigbreak

\section{Introduction and main results}\label{Introduction and main results}
Hilbert $XVI^{th}$ problem represents one of the open question in mathematics and it has produced an impressive amount of
publications throughout the last century. The study of this problem in the context of equivariant dynamical systems is a relatively
new branch of analysis and is based on the development within the last twenty years of the theory of Golubitsky, Stewart and Schaeffer
 in \cite{GS85,GS88}. Other authors \cite{Che} have  specifically
  considered this theory when studying the limit cycles and
related phenomena in systems with symmetry. Roughly speaking the presence of symmetry may complicate the bifurcation analysis because
it often forces eigenvalues of high multiplicity.
 This is not the case of planar systems; on
the contrary, it simplifies the
analysis because of the reduction to  isotypic components.
More precisely  
it allows  us to reduce the bifurcation analysis  to a
region of the complex plane.

In this paper we analyze the $\mathbb{Z}_6-$equivariant system
\begin{equation}\label{main equation}
\dot{z}=\displaystyle{\frac{dz}{dt}}=(p_1+ip_2)z^2\bar{z}+(s_1+is_2)z^3\bar{z}^2-\bar{z}^{5}
 =f(z)  ,
\end{equation}
where $p_1,p_2,s_1,s_2\in\mathbb{R}.$

The general form of the $\mathbb{Z}_q-$equivariant equation is
\[
 \dot{z}=zA(|z|^2)+B\bar{z}^{q-1}+O(|z|^{q+1}),
\]
where $A$ is a polynomial of degree $[(q-1)/2].$ The study of this class of
 equations is developed in several books, see \cite{arn, Che}, when the resonances are
strong, {\it i.e.} $q<4$ or weak $q>4.$ The special case $q=4$ is also treated in several other articles,
see \cite{Rafel1, Che, Zegeling}. In these mentioned
works it is said that the weak resonances are easier to study than the other cases, as the equivariant term $\bar{z}^ {q-1}$ is not dominant with respect to the
function  on $\bar{z}^ 2.$ This is true if the  interest lies in obtaining 
a bifurcation diagram   near 
the origin, but it is no  longer 
true if the  analysis is global and involves the study of limit cycles.
 This is the goal of the present work: studying the global phase portrait of system \eqref{main equation} paying special attention to the existence, location
and uniqueness of limit cycles surrounding 1, 7 or 13 critical points. As far as we know this is the first work in which the existence of limit cycles is studied for
this kind of systems.

The main result of our paper is the following.
\begin{teo}\label{teorema principal}
 Consider equation \eqref{main equation} with $p_2\neq 0$, $|s_2|>1$ 
 and define the quantities:
 \begin{equation*}
\Sigma_A^-=\frac{p_2s_1s_2-\sqrt{p_2^2(s_1^2+s_2^2-1)}}{s_2^2-1},
\qquad
\Sigma_A^+=\frac{p_2s_1s_2+ \sqrt{p_2^2(s_1^2+s_2^2-1)}}{s_2^2-1}.
\end{equation*}
Then, the following statements are true:
\begin{itemize}
\item[(a)] If one of the conditions
\begin{equation*}
(i)\quad p_1\notin\left(\Sigma_A^-,\Sigma_A^+\right),
\qquad
(ii)\quad
p_1\notin\left(\frac{\Sigma_A^-}{2},\frac{\Sigma_A^+}{2}\right)
\end{equation*}
is satisfied, then equation \eqref{main equation} has at most one limit cycle surrounding the origin. 
Furthermore, when the limit cycle exists it is hyperbolic.
\item[(b)] 
There are equations \eqref{main equation} under condition  $(ii)$ having exactly one limit cycle surrounding either $1,~7$ or $13$
critical points, and equations \eqref{main equation} under condition $(i)$ having exactly one hyperbolic limit cycle surrounding either $7$ critical points if $p_1\ne\Sigma_A^\pm$, or the only critical point if $p_1=\Sigma_A^\pm$.
\end{itemize}
\end{teo}

\begin{figure}[ht]
\centering
\begin{center}
\includegraphics[scale=.8]{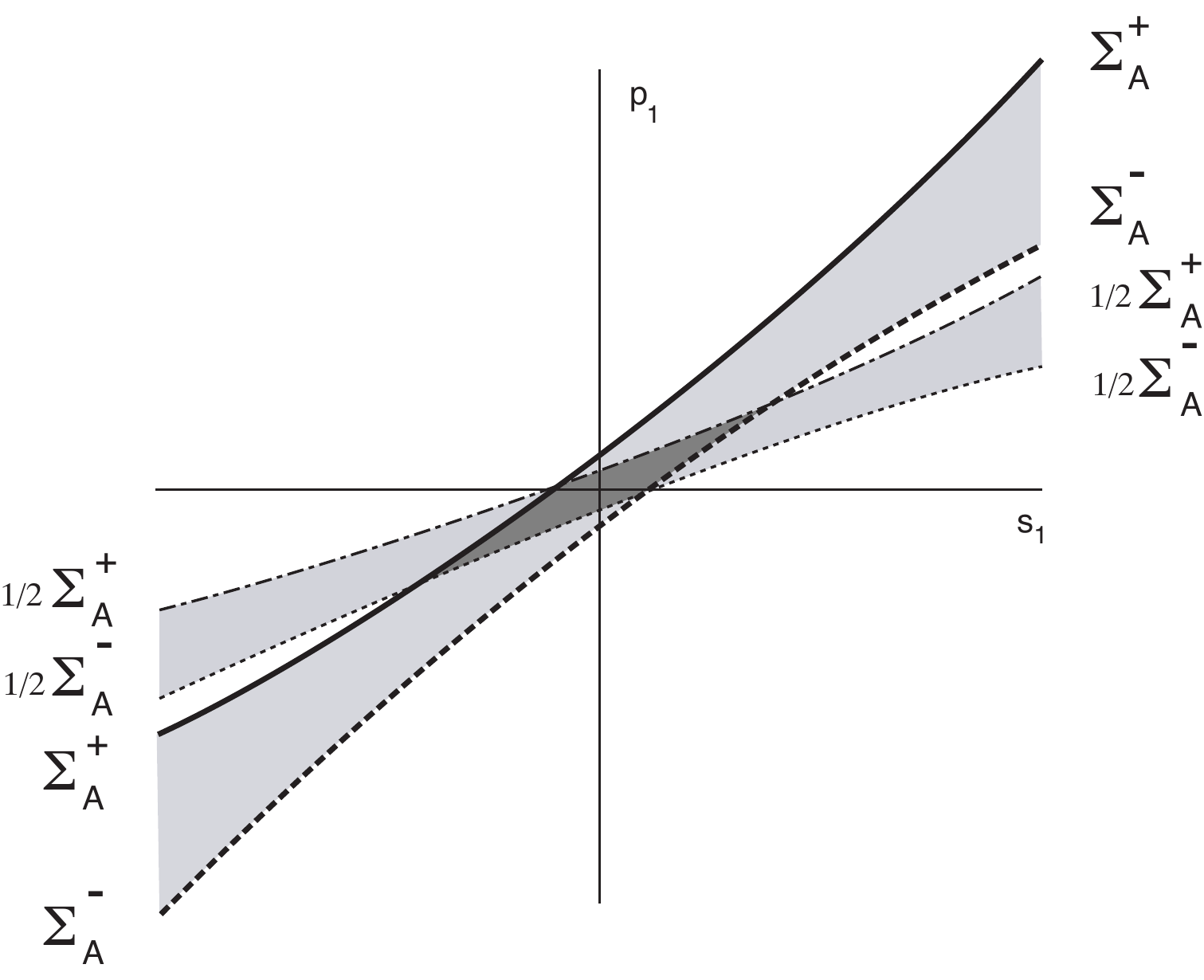}
\caption{Equation \eqref{main equation} has at most one limit cycle surrounding the origin
for $(s_1,p_1)$ outside the dark intersection of the two shaded areas, when $p_2=1$ and $s_2=4$.
Solid line is $\Sigma_A^+$, dashed line stands for $\Sigma_A^-$, dotted line corresponds to 
${\Sigma_A^-}/{2}$ and  dashed-dotted to ${\Sigma_A^+}/{2}$.
For $p_1$ outside the interval $\left(\Sigma_A^-,\Sigma_A^+\right)$ there are at most 7 equilibria, but when $p_1$ lies in that interval, there may be 13 equilibria surrounded by a limit cycle.}\label{FigSigmas}
\end{center}
\end{figure}

 The conditions $(i)$ and $(ii)$  of Theorem~\ref{teorema principal} hold except when $p_1$ lies in 
$\left(\Sigma_A^-,\Sigma_A^+\right)\cap \left(\frac{\Sigma_A^-}{2},\frac{\Sigma_A^+}{2}\right)$.
The intersection of the two intervals is often empty, but this is not always the case as the example of Figure~\ref{FigSigmas} shows. 
Examples of the relative positions of the two intervals are also given in Figure~\ref{A-B-cp} below.

Our strategy for proving Theorem~\ref{teorema principal} will be  to transform the system 
 \eqref{main equation} into a scalar Abel equation and to study it.
Conditions $(i)$ and $(ii)$ define regions where one of the functions in the Abel equation does not change sign. Since these functions correspond to derivatives of the Poincar\'e return map, this imposes an upper bound on the number of limit cycles.
When either $p_1$ lies in the intersection of the two intervals, or $|s_2|\le 1$, this  analysis is not conclusive, and the study of the equations would require other methods.
The qualitative meaning of conditions of statement (a) of the previous theorem is also briefly explained
 and illustrated  in Remark \ref{remark3}  and in
  Figure~\ref{A-B-cp} below.

The paper is organized as follows.
In Section \ref{Preliminary results} we state some preliminary results while in Section \ref{Analysis of the critical points} the study of
the critical points is performed.
Section \ref{Proof of Theorem $1.$} is entirely devoted to the proof of the main theorem of the paper.

\section{Preliminary results}\label{Preliminary results}

We start by obtaining the symmetries of  \eqref{main equation}.
Following \cite{GS85,GS88}, 
a system of differential equations $dx/dt=f(x)$ is said to be $\Gamma-$equivariant if it commutes with the group action of
$\Gamma,$ ie. $f(\gamma x)=\gamma f(x),~\forall\gamma\in\Gamma.$
  Here $\Gamma=  \mathbb{Z}_6$ with the standard action on $\mathbb{C}$ generated by $\gamma_1=\exp(2\pi i/6)$ acting by complex multiplication.
Applying this concept to equation \eqref{main equation} we have the following result.
\begin{prop}
Equation \eqref{main equation} is $\mathbb{Z}_6-$equivariant.
\end{prop}
\begin{proof}
Let $\gamma_k=\exp(2\pi ik/6),~k=0,\ldots,5.$ Then
 a simple calculation shows that $f(\gamma_kz)= \gamma_k f(z)$.
This is true because the monomials in $z$, $\bar{z}$ appearing in the expression of $f$ are the following:
$\bar{z}^5$, which is $\gamma_k$-equivariant, and monomials of the form $z^{\ell+1}\bar{z}^\ell$,
 that are $\mathbb{Z}_n-$equivariant for all $n$.
\end{proof}

Equation \eqref{main equation} represents a perturbation of a Hamiltonian  one and in the following we identify conditions that some parameters have
to fulfill in order to obtain this Hamiltonian. We have the following result.

\begin{teo}\label{teo-Ham}
The Hamiltonian part of $\mathbb{Z}_6-$equivariant equation \eqref{main equation}
is $$\dot{z}=i(p_2+s_2z\bar{z})z^2\bar{z}-\bar{z}^{5}.$$
\end{teo}
\begin{proof} An equation $\dot{z}=F(z,\bar{z})$ is Hamiltonian if $\frac{\partial F}{\partial z}+\frac{\partial F}{\partial \bar{z}} =0.$ For
equation \eqref{main equation} we have
\begin{equation*}\label{hamilt2}
    \begin{array}{l}
    \displaystyle{\frac{\partial F}{\partial z}=2(p_1+ip_2)z\bar{z}+3(s_1+is_2)z^2\bar{z}^2}\\
    \\
    \displaystyle{\frac{\partial \bar{F}}{\partial \bar{z}}=2(p_1-ip_2)z\bar{z}+3(s_1-is_2)z^2\bar{z}^2}
    \end{array}
\end{equation*}
and consequently it is Hamiltonian if and only if $p_1=s_1=0.$
\end{proof}

As we have briefly said in the introduction,  we reduce the study of system \eqref{main equation} to the analysis of a scalar equation of
Abel type. The first step consists in converting equation \eqref{main equation} from cartesian into polar coordinates.

\begin{lema}\label{lema abelian}
The study of periodic orbits  of equation \eqref{main equation}   that surround the origin,
 for $p_2\ne 0$, reduces to the study of
non contractible solutions that satisfy
$x(0)=x(2\pi)$ of the Abel equation
\end{lema}

\begin{equation}\label{abelian0}
    \begin{array}{l}
    \displaystyle{\frac{dx}{d\theta}=A(\theta)x^3+B(\theta)x^2+C(\theta)x}
    \end{array}
\end{equation}
where
\begin{equation}\label{abelian1}
    \begin{array}{l}
    \displaystyle{A(\theta)=\frac{2}{p_2}\left(p_1-p_2s_1s_2+p_1s_2^2+(2p_1s_2-p_2s_1)\sin(6\theta)\right)+}\\
    \hspace{1.3cm}+\displaystyle{\frac{2}{p_2}\left((p_2\sin(6\theta)-p_1\cos(6\theta)+p_2s_2)\cos(6\theta)\right),}\\
    \\
    \displaystyle{B(\theta)=\frac{2}{p_2}\left(p_2s_1-2p_1s_2-p_2\cos(6\theta)-2p_1\sin(6\theta) \right),}\\
    \\
    \displaystyle{C(\theta)=\frac{2p_1}{p_2}.}\\
    \\
    \end{array}
\end{equation}
\begin{proof}
Using the change of variables
\begin{equation*}
z=\sqrt{r}(\cos(\theta)+i\sin(\theta))
\end{equation*}
and  the time rescaling, $\frac{dt}{ds}=r,$ it follows that the solutions of equation \eqref{main equation}
are equivalent to those of the polar system
\begin{equation}\label{polar1}
\left\{
    \begin{array}{l}
    \displaystyle{\dot{r}=2rp_1+2r^2\left(s_1-\cos (6\theta)\right)}\\
    \displaystyle{\dot{\theta}=p_2+r\left(s_2+\sin (6\theta)\right)}
    \end{array}.
    \right.
\end{equation}
From equation \eqref{polar1} we obtain
\begin{equation*}
    \begin{array}{l}
    \displaystyle{\frac{dr}{d\theta}=\frac{\displaystyle{2rp_1+2r^2\left(s_1-\cos (6\theta)\right)}}{p_2+r\left(s_2+\sin (6\theta)\right)}.}
    \end{array}
\end{equation*}
Then we apply the Cherkas transformation $x=\displaystyle{\frac{r}{p_2+r\left(s_2+\sin (6\theta)\right)}}$, see \cite{Cherkas},
 to get the scalar equation \eqref{abelian0}. Obviously the limit cycles that surround the origin of equation \eqref{main equation} are transformed
into non contractible periodic orbits of
equation \eqref{abelian0}, as they cannot intersect the set $\{\dot{\theta}=0\}.$ For more details see \cite{CGP}.
\end{proof}

As we have already mentioned in the introductory section, our goal in this work is to apply the methodology developed in \cite{Rafel1} to study
conditions for existence, location and unicity of the limit cycles surrounding $1,~7$ and $13$ critical points.

A natural way for proving the existence of a limit cycle is to show that,
in the Poncar\'e compactification, infinity  has no critical
points and both infinity and  the origin have the same stability. Therefore, we would like to find the  sets of
parameters for which
 these conditions are satisfied. In the following lemma  we determine
 the stability of  infinity.
\begin{lema}\label{lema characterization origin}
Consider equation \eqref{main equation}  in the Poincar\'e compactification of the plane.
Then:
\begin{itemize}
\item [(i)]   There are no critical points at infinity
if and only if  $|s_2|>1;$
\item [(ii)] When $s_2>1,$  infinity is an attractor (resp. a repellor) when $s_1>0$ (resp. $s_1<0$) and the opposite
when $s_2<-1.$
\end{itemize}
\end{lema}

\begin{proof}
The proof follows the same steps as the Lemma $2.2$ in \cite{Rafel1}. After the change of variable $R=1/r$  in system \eqref{polar1}
and reparametrization
$\displaystyle{\frac{dt}{ds}=R},$ we get the system
\begin{equation*}
\left\{
\begin{array}{l}
R'=\frac{dR}{ds}=\displaystyle{-2R(s_1-\cos(6\theta))-2p_1R^2},\\
\\
\theta'= \frac{d\theta}{ds}=\displaystyle{s_2+\sin(6\theta)+p_2R},
\end{array}
\right.
\end{equation*}
giving the invariant set $\{R=0\},$  that corresponds to the infinity of system \eqref{polar1}. Consequently, it has no critical points at infinity
if and only if $|s_2|>1.$ To compute the stability of infinity in
this case,  we follow \cite{Lloyd} and study the stability of $\{R=0\}$ in the
  system above. 
 This stability is given by the sign of
\begin{equation*}
\displaystyle{\int_0^{2\pi}\frac{-2(s_1-\cos 6\theta)}{s_2+\sin(6\theta)}d\theta}=\displaystyle{\frac{-\hbox{sgn}(s_2)4\pi s_1}{\sqrt{s_2^2-1}}},
\end{equation*}
 and the result follows.
\end{proof}

\section{Analysis of the critical points}\label{Analysis of the critical points}

In this section we are going to analyze which conditions must be satisfied
 to ensure   that equation \eqref{polar1} has one, seven or thirteen critical points.
Obviously, the origin of the system is always a critical
  point. We are going to prove that it is \emph{monodromic}: there is no trajectory of the differential equations that approaches the critical
point with a definite limit direction.

For this purpose, let us define the generalized Lyapunov
constants. Consider the solution of the   following 
scalar equation
\begin{equation}\label{def lyap constants1}
\displaystyle{\frac{dr}{d\theta}=\sum_{i=1}^\infty}R_i\left(\theta\right)r^i,
\end{equation}
where $R_i(\theta),~i\geqslant1$ are $T-$periodic functions. To define the generalized Lyapunov constants, consider the solution of
\eqref{def lyap constants1}  that for $\theta=0$ passes through $\rho.$ It may be written as
\begin{equation*}
\begin{array}{l}
\displaystyle{r(\theta,\rho)=\sum_{i=1}^{\infty}u_i(\theta)\rho^i}
\end{array}
\end{equation*}
with $u_1(0)=1,~u_k(0)=0, \forall~k\geqslant2.$
Hence, the return map of this solution is given by the series
\begin{equation*}
\begin{array}{l}
\displaystyle{\Pi(\rho)=\sum_{i=1}^{\infty}u_i(T)\rho^i}.
\end{array}
\end{equation*}
  For a given 
system, in order to determine the stability of a solution, the only significant term in the return map is the first nonvanishing term that
makes it differ from the identity map. Moreover, this term  will determine the stability of
this solution. On the other hand, if we consider a family of systems depending
on parameters, each of the $u_i(T)$ depends on these parameters. We will call
$k^{th}$ generalized Lyapunov constant $V_k=u_k(T)$ the value of this expression
assuming $u_1(T) = 1, u_2(T) =\ldots,= u_{k-1}(T) = 0.$

\begin{lema}\label{lema Lyapunov}
The origin of system \eqref{polar1} is  monodromic, if $p_2\ne 0$.
Moreover, its  stability if given by the sign of $p_1,$ if it is
not zero, and by the sign of $s_1$ if $p_1=0$.
\end{lema}

\begin{proof}
To prove that the origin is monodromic we calculate the
arriving directions of the flow to the origin, see \cite[Chapter IX]{Andronov} for more details. Concretely if we
write system \eqref{polar1} 
in cartesian coordinates we get
\begin{equation}\label{cartesian}
\begin{array}{ll}
\dot{x}=P(x,y)= &\Big(p_1x^3-p_2x^2y+p_1xy^2-p_2y^3\Big)+(s_1-1)x^5-s_2x^4y\\
&+(2s_1+10)x^3y^2-2s_2x^2y^3+(s_1-5)xy^4-s_2y^5\\
&\\
\dot{y}=Q(x,y)=&\Big(p_2x^3+p_1x^2y+p_2xy^2+p_1y^3\Big)+s_2x^5+(s_1+5)x^4y\\
&+2s_2x^3y^2+(2s_1-10)x^2y^3
  +s_2xy^4+(s_1+1)y^5,
\end{array}
\end{equation}
  Then any solutions arriving at the origin will be tangent to the directions $\theta$ that are the zeros of
$r \dot \theta=R(x,y)=xQ(x,y)-yP(x,y)$.  
Since the lowest order term of $R(x,y)$ is $p_2(x^2+y^2)^2$,
which is always different from zero away from the origin, 
it follows that the origin is either a focus or a center.

To know the stability of the origin, we compute the two first  Lyapunov constants. As it is well-known, the two first Lyapunov constants of an
Abel equation are given by$$
V_1=\displaystyle{\exp\left(\int_0^{2\pi}C(\theta)d\theta\right)}-1,\qquad\qquad
V_2=\displaystyle{ \int_0^{ 2\pi }B(\theta)d\theta}.
$$
Applying this to equation \eqref{abelian0} with the expressions given in \eqref{abelian1} we get the following result:
$$
V_1=\displaystyle{\exp\left(4\pi\frac{p_1}{p_2}\right)}-1,
$$
 and if $V_1=0$, then
$
V_2=4\pi s_1,
$
and we get the result.
\end{proof}

\begin{nota}
It is clear that $(V_1,V_2)=(0,0)$ if and only if $(p_1,s_1)=(0,0)$ i.e. equation \eqref{main equation} is hamiltonian (see Theorem \ref{teo-Ham}).
Consequently, as the origin remains being monodromic, in this case it is a center.
\end{nota}

In the next result we study the  equilibria 
of equation \eqref{polar1}   with $r\ne 0$. 
They will be the non-zero critical points of the system.
\begin{lema}\label{lema solutions r theta}
Let $-\pi/6<\theta<\pi/6.$ Then the equilibria of system \eqref{polar1}  with $r\ne 0$ are given by:
\begin{equation}\label{solution r theta }
r={\frac{-p_2}{s_2+\sin\left(2\theta_{\pm}\right)}, 
\qquad\qquad
\theta_{\pm}=\frac{1}{3}\arctan(\Delta_{\pm})},
\end{equation}
where 
$
\Delta_{\pm}=\displaystyle\frac{p_1\pm u}{p_2-p_1s_2+p_2s_1}$
and 
$
u=\sqrt{p_1^2+p_2^2-\left(p_1s_2-p_2s_1\right)^2}$.
\end{lema}

\begin{proof}
Let $-\pi/6<\theta<\pi/6.$
To compute the critical points of system \eqref{polar1}, we have to solve the following nonlinear system:
\begin{equation}\label{critical points0}
    \begin{array}{l}
    \displaystyle{0=2rp_1+2r^2\left(s_1-\cos (6\theta)\right)}\\
    \displaystyle{0=p_2+r\left(s_2+\sin (6\theta)\right).}
    \end{array}
\end{equation}

Let $x=6\theta$ and $t=\tan\frac{x}{2},$ so $t=\tan3\theta.$ Then, doing some simple computations one gets
\begin{equation}\label{Trigo}
\sin x=\displaystyle{\frac{2t}{1+t^2}}\qquad
\cos x=\displaystyle{\frac{1-t^2}{1+t^2}}.
\end{equation}
Eliminating $r$ from equations \eqref{critical points0}
and using the previous formulas we get
\begin{equation*}
(-p_2+p_1s_2-p_2s_1)t^2+2p_1t+p_2+p_1s_2-p_2s_1=0.
\end{equation*}
Solving the previous equation for $t,$ yields the result.

Finally, consider the interval $-\pi/6\leqslant\theta<\pi/6,$  let $x=6\theta$ and $\tau=\cot\frac{x}{2},$ so $\tau=\cot3\theta.$
The same expressions \eqref{Trigo} for $\sin x$ and $\cos x$ hold if $t$ is replaced by $\tau=1/t$, hence the rest of the proof is applicable.
\end{proof}

We will prove now that simultaneous equilibria  of the type $(r,\theta)=(r,0)$, $(r,\theta)=(\tilde{r},\pi/6)$ 
are not possible.

\begin{lema}
If  $|s_2|>1,$ then there are no parameter values for which there are  simultaneous equilibria
of  system \eqref{polar1} for $\theta=0$ and $\theta=\pi/6$ different from
the origin.
\end{lema}
\begin{proof}
If we solve
\begin{equation*}
\left\{
\begin{array}{l}
0=p_1+r(s_1-\cos6\theta)\\
0=p_2+r(s_2+\sin6\theta)
\end{array}
\right.
\end{equation*}
for $\theta=0,$ we get $p_1s_2=p_2(s_1-1)$ with the restriction
 $\sign p_2=- \sign s_2$, 
 to have $r$ well defined in the
second equation. On the other hand,
solving the same system for $\theta=\pi/6,$ we get $p_1s_2=p_2(s_1+1)$ with the same restriction $\sign p_2=- \sign s_2.$
This means $p_2=0$ that implies $r=0$  and the result follows.
\end{proof}

In the following we sumarize the conditions that parameters have to fulfill in order that system \eqref{polar1} has exactly
one, seven or thirteen critical points  (see figure~\ref{BD}).

\begin{lema}\label{7critical points}
Consider system \eqref{polar1} with  $|s_2|>1.$ If $s_2p_2\ge 0$, 
then the only equilibrium is the origin.
 If $s_2p_2<0$ 
then the number of equilibria of  system \eqref{polar1} is determined by the quadratic form:
\begin{equation}\label{quadraticForm}
{\mathcal Q}(p_1,p_2)=p_1^2+p_2^2-(p_1 s_2-p_2 s_1)^2=(1-s_2^2)p_1^2+(1-s_1^2)p_2^2+2s_1s_2p_1p_2
\end{equation}
Concretely:
\begin{enumerate}
\item \label{oneEq}
exactly one equilibrium (the origin) if ${\mathcal Q}(p_1,p_2)<0$;
\item \label{sevenEq}
exactly seven equilibria  (the origin and one non-degenerate
saddle-node per sextant)  if ${\mathcal Q}(p_1,p_2)=0$;
\item \label{thirteenEq}
exactly thirteen equilibria  (the origin and two per sextant, a saddle and a node) if ${\mathcal Q}(p_1,p_2)>0$.
\end{enumerate}
\end{lema}

\begin{figure}[ht]
\centering
\begin{center}
\includegraphics[scale=1]{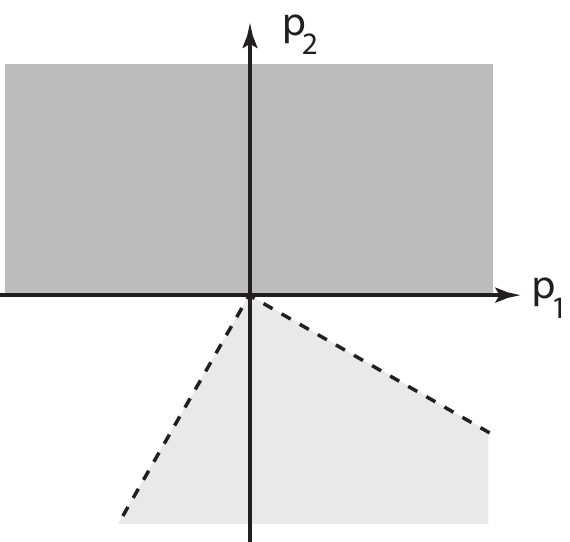}\caption{ Bifurcation diagram for equilibria of equation \eqref{polar1} with $s_2>1$
on the $(p_1,p_2)$ plane. For $(p_1,p_2)$ in the shaded regions, the only equilibrium is the origin. In the white region there are two other
equilibria, a saddle and a node, in each sextant. On the dotted line, the saddle and the node come together into a saddle-node  in each sextant.
When $p_2$ tends to 0 in the  white region, the two equilibria tend to the origin.
 For $s_2<-1$, the diagram is obtained by reflecting on the $p_1$ axis.}\label{BD}
\end{center}
\end{figure}

\begin{proof}
We will do the proof for the case $s_2>1,$  the other case being analogous.

There will exist more critical points different from the origin if and only if the formulas  \eqref{solution r theta } given in
Lemma \ref{lema solutions r theta} are realizable  with $r>0$.
This will not occur either when  the expression for $r$ is negative  (that corresponds to  the condition $p_2\ge 0$) or
when  the discriminant in  $\Delta_{\pm}$ is negative (that corresponds to ${\mathcal Q}(p_1,p_2)<0$).

In order to have exactly six more critical points, two conditions have to be satisfied: $p_2<0$ to ensure positive values
of $r,$ as $s_2>1$, and the quantities $\Delta_{\pm}$ have to coincide, {\it i.e.}
the discriminant $u$ in Lemma \ref{lema solutions r theta} has to be zero, that is the condition (\ref{sevenEq}) given in the statement of the lemma.
Hence, $r_+=r_-$ and $\theta_+=\theta_-$.

We prove now that these additional critical points are saddle-nodes. 
By symmetry, we only need to prove it for one of them.
The Jacobian matrix of system \eqref{polar1} evaluated at $(r_+,\theta_+)$ is
\begin{equation}\label{jacobian general}
J_{(r_+,\theta_+)}=
\left(\begin{matrix}
            2p_1+4r_+(s_1-\cos(6\theta_+)&12r_+^2\sin(6\theta_+)\\
            s_2+\sin(6\theta_+)&6r_+\cos(6\theta_+)\\
        \end{matrix}\right). 
\end{equation}

Evaluating the Jacobian matrix \eqref{jacobian general} at the concrete expression  \eqref{solution r theta } of the critical point and taking into account the
condition $p_1^2+p_2^2=(p_1s_2-p_2s_1)^2,$ the eigenvalues of the matrix are
$$
\lambda_1=0
\qquad\qquad
\lambda_2=2p_1-\displaystyle{\frac{6p_2(p_2+p_2s_1-p_1s_2)}{(p_2s_2+p_1)(1+s_1)-p_1s_2^2}}.
$$

Therefore,
$(r_+,\theta_+)$ has a zero eigenvalue. To show that these critical points are saddle-nodes  we use the following reasoning.
It is well--known, see  for instance \cite{Andronov},  that the sum of the indices of all critical points contained in the interior of a limit cycle of a planar system
is $+1.$ As under our hypothesis the infinity does not have critical points on it, it is a limit cycle of the system and it has seven singularities
in its interior:
the origin, that is a focus and hence has index +1,
and 6 more critical points, all of the same type because of the symmetry. Consequently, the index of these critical points must be 0. As we have
proved that they are semi-hyperbolic critical points then they must be saddle-nodes.

In order that equation \eqref{polar1} has exactly thirteen critical points it is enough that $r>0$ ({\it i.e.} $p_2<0$) and
the discriminant in $\Delta_{\pm}$ of Lemma~\ref{lema solutions r theta} is positive, that
is the condition (\ref{thirteenEq}) of the statement.

To get the stability of the twelve critical points  we evaluate  the Jacobian matrix \eqref{jacobian general} 
at these critical points, 
and taking into account the condition
$p_1^2+p_2^2>(p_1s_2-p_2s_1)^2,$ the eigenvalues of the critical points
$(r_+,\theta_+)$ are $\lambda_{1,2}=R_+\pm S_+,$
while the eigenvalues of $(r_-,\theta_-)$ are $\alpha_{1,2}=R_-\pm S_-,$ where
\begin{equation*}
\begin{array}{l}
R_{\pm}=\displaystyle{\frac{p_1^2s_1\pm3p_2^2s_1\mp2p_1p_2s_2+2p_1u}{\mp p_1s_1\mp p_2s_2+u}},\\\\
S_{\pm}=\displaystyle{\frac{\sqrt{48(p_1^2+p_2^2)u(u\mp p_1s_1\mp p_2s_2)+(2p_1^2s_1+6p_2^2s_1-4p_1p_2s_12\pm 4p_1u)^2}}{p_1\mp p_1s_1\mp p_2s_2} }.
\end{array}
\end{equation*}

In the following we will show that one of the critical points has index $+1,$ while the other is a saddle.

Doing some computations one gets that the product of the eigenvalues of $(r_+,\theta_+)$ and $(r_-,\theta_-)$ are, respectively
\begin{equation*}
\begin{array}{l}
R_+^2-S_+^2=\displaystyle{\frac{-12(p_1^2+p_2^2)u}{u-p_1s_1-p_2s_2}}          \\
\\
R_-^2-S_-^2=\displaystyle{\frac{-12(p_1^2+p_2^2)u}{u+p_1s_1+p_2s_2}} .
\end{array}
\end{equation*}
The numerator of both expressions is negative.

If $p_1s_1+p_2s_2>0$ then $S_-^2>0$ and $R_-^2-S_-^2<0$. Consequently the critical point $(r_-,\theta_-)$ is a saddle.
On the other hand, the denominator of $R_+^2-S_+^2$  is negative. This is true as
\begin{eqnarray*}
 &&u-p_1s_1-p_2s_2<0\iff u^2< (p_1s_1+p_2s_2)^2\iff \\
&&\iff p_1^2+p_2^2-p_2^2s_1^2-p_1^2s_2^2+2p_1p_2s_1s_2<p_1^2s_1^2+p_2^2s_2^2+2p_1p_2s_1s_2\\
&&\iff 0<(p_1^2+p_2^2)(s_1^2+s_2^2-1),
\end{eqnarray*}
that is always true.  Hence, $R_+^2-S_+^2>0$ and $(r_+,\theta_+)$ has index $+1.$

If $p_1s_1+p_2s_2<0,$ doing similar reasonings one gets that the critical point $(r_+,\theta_+)$ is a saddle
while $(r_-,\theta_-)$ has index $+1.$
\end{proof}

Note that  ${\mathcal Q}$ is a quadratic form on $p_1,p_2$, and its determinant $1-s_1^2-s_2^2,$  is negative if
$s_2>1.$
Hence, for each choice of $s_1, s_2$ with $s_2>1$, the
points where ${\mathcal Q}(p_1,p_2)$ is negative lie on two sectors, delimited by the two perpendicular lines where ${\mathcal Q}(p_1,p_2)=0$.
Since ${\mathcal Q}(0,p_2)=(1-s_2^2)p_2^2<0$ for $s_2>1$, then the sectors where there are two equilibria in each sextant do not
include the $p_2$ axis, as in Figure~\ref{BD}.

\begin{lema}\label{signA}
Consider $|s_2|>1$ and define the following two numbers:
\begin{equation*}
\Sigma_A^-=\frac{p_2s_1s_2- \sqrt{p_2^2\left(s_1^2+s_2^2-1\right)}}{s_2^2-1},
\hspace{0.3cm}\Sigma_A^+=\frac{p_2s_1s_2+ \sqrt{p_2^2\left(s_1^2+s_2^2-1\right)}}{s_2^2-1}.
\end{equation*}
Let $A(\theta)$ be the function given in Lemma \ref{lema abelian}. Then the function $A(\theta)$ changes sign if and only if
$p_1\in\left(\Sigma_A^-,\Sigma_A^+\right).$
\end{lema}
\begin{proof}
Writing $x=\sin(6\theta),~y=\cos(6\theta)$, the function $A(\theta)$ in \eqref{abelian1} becomes
\begin{equation*}
A(x,y)=\displaystyle{\frac{2}{p_2}\left(p_1-p_2s_1s_2+p_1s_2^2+(2p_1s_2-p_2s_1)x+  (p_2x-p_1y+p_2s_2)y\right)}.
 \end{equation*}
Next we solve the set of equations
\begin{eqnarray*}
 A(x,y)=0,\\
x^2+y^2=1.
\end{eqnarray*}
to get the solutions
\begin{equation*}
\begin{array}{l}
x_1=-s_2,~y_1=\sqrt{1-s_2^2}\\
x_2=-s_2,~y_2=-\sqrt{1-s_2^2}\\\\
x_\pm=\displaystyle{\frac{p_1p_2s_1-p_1^2s_2\pm p_2\sqrt{p_1^2+p_2^2-(p_2s_1-p_1s_2)^2}}{p_1^2+p_2^2}}\\
\\
y_\pm=\displaystyle{\frac{p_2^2s_1-p_1p_2s_2\mp p_1 \sqrt{p_1^2+p_2^2-(p_2s_1-p_1s_2)^2}}{p_1^2+p_2^2}}\\
\end{array}
\end{equation*}
Observe now the the two first pairs of solutions $(x_1,y_1), (x_2,y_2)$ cannot be solutions of our equation $A(\theta)=0$ since $x=\sin(6\theta)=-s_2<-1.$

On the other hand, if we look for the intervals where the function $A(\theta)$ does not change sign we have two possibilities:
either $|x_\pm|>1$ (and consequently $|y_\pm|>1$) or  the discriminant of $x_\pm,$ $\Delta=p_1^2+p_2^2-(p_2s_1-p_1s_2)^2$ is negative or zero. In the  case
$\Delta<0$,
the solutions will be complex non-real,  and in the second case $\Delta=0$, the function will have a zero but a double one, and it will not
change its sign.

The first possibility turns out to be impossible in our region of parameters. The second possibility leads
to the region $p_1\in\mathcal{R}\setminus (\Sigma_A^+,\Sigma_A^-).$
\end{proof}

\begin{lema}\label{sigmalemaB}
Consider $|s_2|>1$ and define the following two numbers:
\begin{equation*}
\Sigma_B^\pm=\frac{\Sigma_A^\pm}{2}.
\end{equation*}
Let $B(\theta)$ be the function given in Lemma \ref{lema abelian}. Then the function $B(\theta)$ changes sign if and only if
$p_1\in\left(\Sigma_B^-,\Sigma_B^+\right).$
\end{lema}

\begin{proof}
 By direct computations with the same substitution as the one in the proof of the previous lemma, $x=\sin(6\theta), y=\cos(6\theta),$
we get that the zeroes of the system
\begin{eqnarray*}
B(x,y)=0,\\
x^2+y^2=1,
\end{eqnarray*}
are
\begin{eqnarray*}
x_\pm= \frac{2p_1p_2s_1-4p_1^2s_2\pm p_2\sqrt{4p_1^2+p_2^2-(p_2s_1-2p_1s_2)^2}}{4 p_1^2+p_2^2},\\
y_{\pm}=\frac{p_2^2s_1-2p_1p_2s_2\mp 2p_1\sqrt{4p_1^2+p_2^2-(p_2s_1-2p_1s_2)^2}}{4 p_1^2+p_2^2}.
\end{eqnarray*}
Applying  arguments similar to those in the previous proof,  we get that the function $B(\theta)$ will not change sign if and only if
$p_1\not\in(\Sigma_B^+,\Sigma_B^-).$

\end{proof}

\begin{nota}\begin{itemize}
\item In general systems there are examples 
for which the function $A(\theta)$ changes sign but $B(\theta)$ does not
and vice-versa.

\item The condition given in Lemma \ref{signA} under which the function $A(\theta)$ does not change sign is closely related to the number of critical points of system \eqref{polar1}. 
 More precisely, the condition $\Delta<0$ in the proof of Lemma \ref{signA} is the same as the condition for 
 existence of a unique critical point in
 system \eqref{polar1},
while the condition $\Delta=0$, together with $p_2<0$, is equivalent to the existence of
exactly 7 critical points. 

 \begin{figure}[h!]
 \includegraphics[width=11cm]{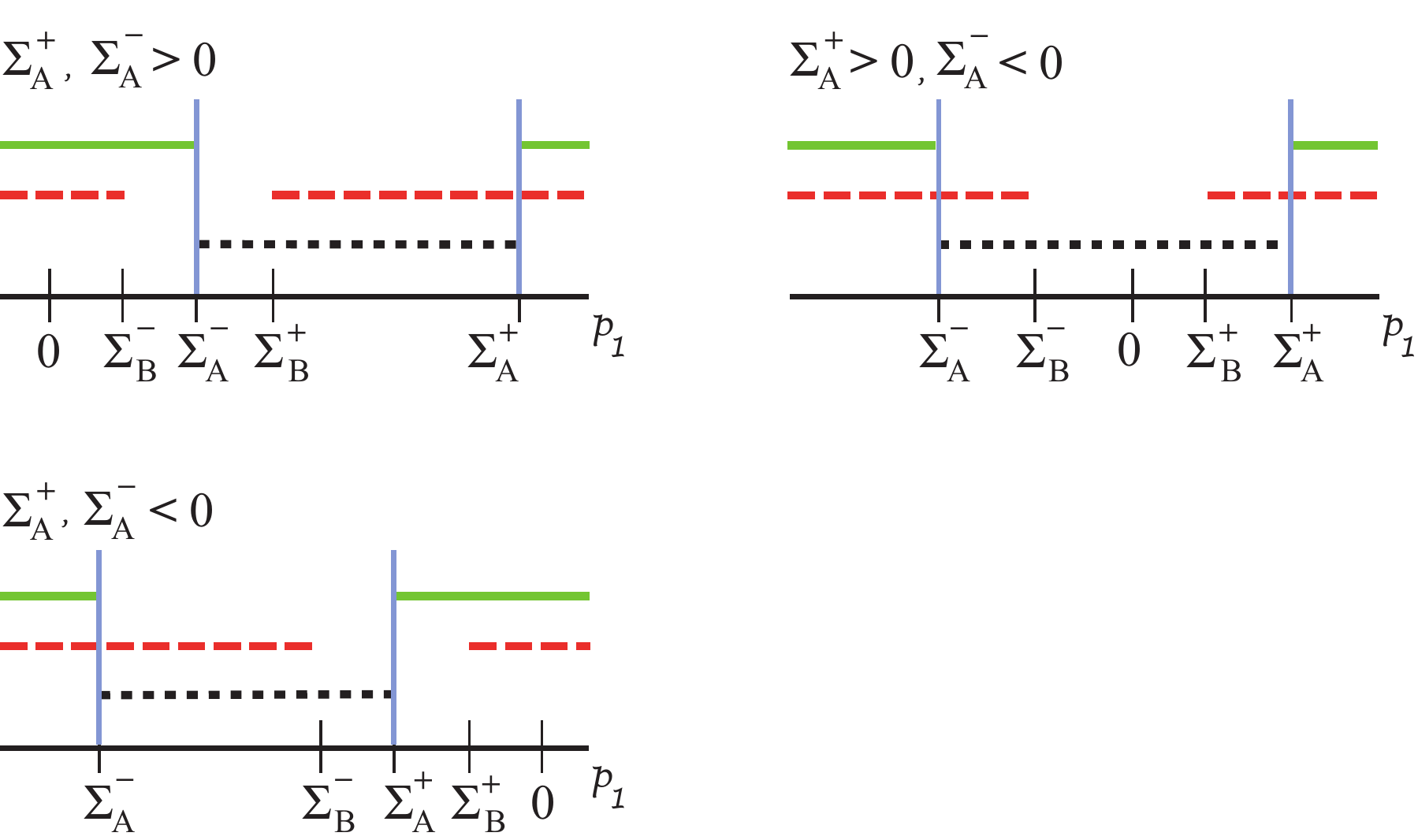}
\caption{A sketch of  intervals where the conditions of Theorem~\ref{teorema principal} are satisfied,  with $s_2>1,p_2<0.$
The result follows from the relationship between the function $A(\theta)$
 and $B(\theta)$ changing sign and the number of critical points of  \eqref{polar1}.
 The green solid line indicates where the function $A(\theta)$  does not change sign, the red dashed one where $B(\theta)$
 does not change sign, while the dotted line represents
 the interval where 13 critical points exist. The blue vertical lines (that correspond to $\Sigma_A^\pm$)
 stand for 7 critical points.
 There is another possibility, not shown here, that the two intervals are disjoint, see Figure~\ref{FigSigmas}.\label{A-B-cp}}
 \end{figure}
\end{itemize}
\end{nota}

In the following we will need some results on Abel equations proved in \cite{pliss} and \cite{Llibre}. We sumarize them in a theorem.

\begin{teo}\label{teorema Llibre}
Consider the Abel equation \eqref{abelian0} and assume that either $A(\theta)\not\equiv0$ or $B(\theta)\not\equiv0$ does not change sign.
Then it has at most three solutions satisfying $x(0)=x(2\pi),$ taking into account their multiplicities.
\end{teo}

\begin{nota}\label{remark3}
Condition (a) ot Theorem \ref{teorema principal} implies that one of the functions $A(\theta),B(\theta)$ of the Abel equation does not change sign.
If condition (i) is satisfied then $A(\theta)$ does not change sign and, as the third derivative of the Poincar\'e return map of the Abel equation is
this function $A(\theta)$ then, the Abel equation can have at most 3 limit cycles. In our case, one of them is the origin and another one is  infinity.
Consequently only one non-trivial limit cycle can exist, both for the scalar equation and for the planar system.

If condition (ii) is verified then it is the function $B(\theta)$ that does not change sign and, by a change of coordinates, one can get 
another scalar equation for which the  third derivative of the Poincar\'e map is the function $B(\theta),$ getting the same conclusions as before.
 \end{nota}

\section{Limit cycles}\label{Proof of Theorem $1.$}

\begin{figure}[ht]
\begin{center}
\includegraphics[width=12cm]{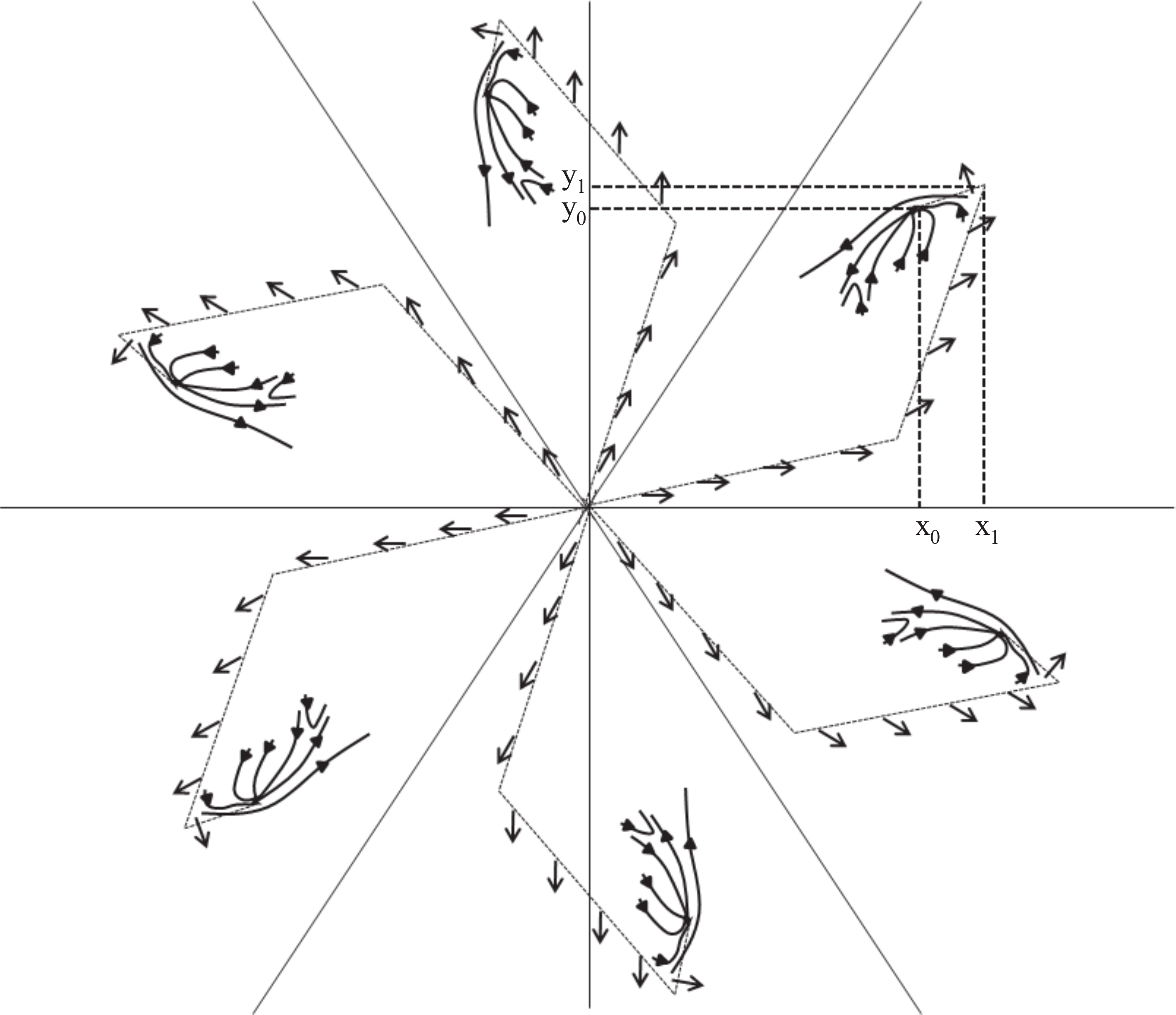}
\caption{The polygonal curve with no contact with the flow of the differential equation and the separatrices of the saddle-nodes of system \eqref{critical points0}.}\label{figure 1}
\end{center}
\end{figure}

\begin{proof}[Proof of Theorem~\ref{teorema principal}]
We first define the function $c(\theta)=s_2+\sin(6\theta)$ and the set $\Theta:=\{(r,\theta):\dot{\theta}=p_2+(s_2+\sin(6\theta))r=0\}.$
Since $|s_2|>1,$ we have $c(\theta)\neq0,~\forall\theta\in[0,2\pi].$

$(a)$ Let's assume first that condition $(i)$ is satisfied. By Lemma \ref{lema abelian}, we reduce the study of the periodic orbits of
equation \eqref{main equation} to the analysis of the  non contractible periodic orbits of the Abel equation \eqref{abelian0}.
Since $p_1\notin\left(\Sigma_A^-,\Sigma_A^+\right),$ by condition $(i)$ of Lemma \ref{signA}, we know that function $A(\theta)$
in the Abel equation does not change sign. Hence, from Theorem \ref{teorema Llibre}, the maximum number of solutions satisfying
$x(0)=x(2\pi)$ in system \eqref{abelian0} taking into account their multiplicities, is three. One of them is trivially $x=0.$
Since $c(\theta)\neq0,$
by simple calculations we can prove that the curve $x=1/c(\theta)$ is a second solution satisfying this condition. As shown in \cite{Rafel1},
undoing the Cherkas transformation we get that $x=1/c(\theta)$ is mapped into infinity of the differential equation. Then, by
Lemma \ref{lema abelian}, the maximum number of limit cycles of equation \eqref{main equation} is one. Moreover, from the same
lemma it follows that the limit cycle is hyperbolic.
 From the symmetry, it follows that a unique limit cycle must surround the origin.

$(b)$ 
We follow the same analysis method as in \cite{Rafel1}. 
When $p_1>\Sigma_A^+,$ by Lemma \ref{lema Lyapunov}, both the origin
and infinity in the Poincar\'e compactification are repelors. In particular the origin is an unstable focus.
On the other hand, from Lemma \ref{signA}, $A(\theta)$ does not vanish  when $p_1>\Sigma_A^+,$ and the origin is the unique critical point.
 It is easy to see that, since $p_2A(\theta)>0,$ then the exterior of the closed curve $\Theta$ is positively invariant and therefore,
by applying the Poincar\'e-Bendixson Theorem and part $(a)$ of this theorem, there is exactly one hyperbolic limit cycle surrounding
the curve $\Theta.$ Moreover, this  limit cycle  is  stable.

When $p_1=\Sigma_A^+,$ six more semi-elementary critical points appear and they are located on $\Theta.$ They are saddle-nodes as shown in Lemma \ref{7critical points}.
We will show that at this value of $p_1$ the periodic orbit still exists and it
surrounds the seven critical points. We will prove this by constructing a polygonal line with no contact with the flow of the differential equation.
On the polygonal, the vector field points outside, and 
consequently, as the infinity is a repelor, the $\omega$-limit set of the unstable separatrices of the saddle-nodes must be a limit cycle surrounding
$\Theta,$ see Figure \ref{figure 1} and Example~\ref{example} below.

The $\mathbb{Z}_6-$equivariance of system \eqref{main equation} allows us to study the flow in only one sextant of the phase space in cartesian
 coordinates, the behaviour in the rest of the phase space being identical.
 The polygonal line will join the origin to one of the saddle-nodes as in  Figure \ref{figure 1}.

We explain the construction of the polygonal line when  $0<s_1\leq 1$, $s_2>1$ and $p_2<0$.
In this case, $p_1=\Sigma_A^+$ implies that  $p_1>0$ and, in the notation of Lemma~\ref{lema solutions r theta},
we have $u=0$. Hence
$$
\frac{1}{\Delta_{\pm}}=\frac{(s_2^2-1)(1+s_1)}{s_1s_2-\sqrt{s_1^2+s_2^2-1}}-s_2<-1 
$$
and using the expression in  Lemma~\ref{lema solutions r theta}, we find that the angular coordinate $\theta_0$ of the saddle-node satisfies $-1<\tan 3\theta_0<0$, and therefore $\pi/4<\theta_0<\pi/3$.

The first segment in the polygonal is the line $\theta=\pi/4$. Using the expression \eqref{polar1}
we obtain that on this line $\dot\theta=p_2+r(s_2-1)$ which is negative for $r$ between $0$ and $r_1=-p2/(s_2-1)>0$. Thus, the vector field is transverse to this segment and points away from the saddle-node on it.

Another segment in the polygonal is obtained using the eigenvector $v$ corresponding to
the non-zero eigenvalue of the saddle-node $z_0$. The line $z_0+tv$ is the  tangent to the separatrix of the hyperbolic region of the saddle-node.
Let $t_0$ be the first positive value of $t$ for which the vector field fails to be transverse to this line.

If the lines  $z_0+tv$ and $\theta=\pi/4$ cross at a point with $0<t<t_0$ and with $0<r<r_1$, then the polygonal consists of the two segments.
If this is not the case, then  the segment joining  $z_0+t_0v$ to $r_1\left(\sqrt{2}/2,\sqrt{2}/2\right)$ will also be transverse to the vector field,
and the three segments will form the desired polygonal.

By using the same
arguments  presented above  and the fact that infinity is a repelor, it follows by the Poincar\'e-Bendixson Theorem  that
the only possible $\omega$-limit for the unstable separatrix of the saddle-node is a periodic orbit which has to surround the six saddle-nodes,
see again Figure \ref{figure 1}.

If $p_1=\Sigma_A^+$ then the function $A(\theta)$ does not change sign; therefore the arguments presented in  part $(a)$ of the proof of this theorem
assure the hyperbolicity of the limit cycle.

When we move $p_1$ towards zero but still very close to $\Sigma_A^+,$ then $B(\theta)$ is strictly positive because $\Sigma_A^+>\Sigma_B^+$ and there
are $13$ critical points as shown in Lemma \ref{7critical points}: the origin (which is a focus), six saddles and six critical points of index $+1$ on
$\Theta.$ Applying one more time part $(a)$ of the proof of this theorem, we know that the maximum number of limit cycles surrounding the origin is one.
If $p_1=\Sigma_A^+$ the limit cycle is hyperbolic and it still exists for the mentioned value of $p_1.$ Then we have the vector field with $B(\theta)$
not changing sign, $12$ non-zero critical points and a limit cycle which surrounds them together with the origin.
\end{proof}

\begin{ex}\label{example}

As an example of the construction of the polygonal line in the proof of Theorem~\ref{teorema principal} we present a particular case, done with Maple.
Let's fix parameters $p_2=-1,~s_1=-0.5,~s_2=1.2.$
With these values of the parameters, we have $\Sigma_A^-=-0.52423$, $\Sigma_A^+=3.25151.$

To construct the polygonal line we work in cartesian coordinates. A key
point in the process of identifying the three segments of the polygonal line is knowing explicitly the eigenvector corresponding to
the non-zero eigenvalue of the
saddle-node $(x_0,y_0)=(1.358,1.5).$ This eigenvector is $v=(-0.8594,-0.5114)$ so  
the slope of the tangent to the hyperbolic direction of the
saddle-node is $0.5114/0.8594$ and the straight line of this slope passing through the saddle-node is
$R\equiv\{y=1.5+\displaystyle{\frac{0.5114}{0.8594}}(x-1.358)\}.$ 
The tangent to the separatrix of the hyperbolic region of the saddle-node is locally transverse to this line.
The scalar product of the vector field associated to
equation \eqref{critical points0} with the normal vector to $R,$ $(0.5114,-0.8594),$ is given by the equation
\begin{equation*}
\begin{array}{l}
-2.39191647949065x^5+2.34410741916533x^4+4.86235167862649x^3-\\
2.71272659052423x^2-2.33924612305747x-0.92289951077311
\end{array}
\end{equation*}
when evaluated on the straight line $R;$ its unique real root is $x\simeq-1.1737$ and the flow is transversal from inside out through $R$ for any $x>-1.1737.$\\
Let us now define the polygonal line as
\begin{equation*}
(x(t),y(t))=
\left\{
\begin{array}{l}
(t,t)\hspace{7.1cm}\mathrm{if}~0\leqslant t<2\sqrt{\frac{1}{2.8}},\\
\\
(t,1.5(t-2\sqrt{\frac{1}{3}})+2\sqrt{\frac{1}{2.8}})\hspace{2.5cm}\mathrm{if}
~2\sqrt{\frac{1}{2.8}}\leqslant t<1.425,\\
\\
((t-2)\displaystyle{\frac{x_1-x_0}{x_1-2}}+x_0,(t-2)\displaystyle{\frac{y_1-y_0}{y_1-2}}+y_0)\hspace{0.5cm}\mathrm{if}
~1.425\leqslant t<2.
\end{array}
\right.
\end{equation*}
Let's now consider the point $(x_1,y_1)=(1.4250,1.5399),$ which stands at the intersection between the second and third segments of the polygonal line.
The scalar product between the normal to each segment and the flow of the differential equation is negative, when evaluated on the corresponding segments.
To show the calculation, we will exemplify it for the first segment, the remaining cases being treated similarly.

We elected the segment of the line $L\equiv\{y=x\}$ and when substituting it into the system in cartesian coordinates \eqref{cartesian} we obtain
\begin{equation*}
\left\{
\begin{array}{l}
\dot{x}=4x^5(s_1-s_2+1)+2x^3(p_1-p_2)\\
\dot{y}=4x^5(4s_1+5s_2-4)+2x^3(p_1+p_2).
\end{array}
\right.
\end{equation*}
A normal vector to the line $L$ is $(-1,1)$ and the scalar product of $(\dot{x},\dot{y})$ with $(-1,1)$ yields $f(x)=x^3(4p_2+x^2(9s_2-8)).$
Solving this last equation leads to $f(x)<0$ for $-2\sqrt{\frac{-p_2}{9s_2-8}}\leqslant x\leqslant 2\sqrt{\frac{-p_2}{9s_2-8}}.$
So we choose $0\leqslant x\leqslant 2\sqrt{\frac{-p_2}{9s_2-8}},$ and we get that this scalar product is negative in the region where the
polygonal line is defined as $(t,t).$

By using the same 
arguments  presented above  
 it follows  that the only possible $\omega$-limit for the unstable separatrix of the saddle-node is a periodic orbit which has to surround the six saddle-nodes,
see again Figure \ref{figure 1}.
\end{ex}

\paragraph{\bf Acknowledgements}
An important part of the work was carried out during a FCT-supported visit of A.C.
Murza to the Departament Ci\`{e}ncies Ma\-te\-m\`{a}\-ti\-ques i Inform\`{a}tica de la Universitat de les Illes Balears, under the supervision of Professors
Mar\'{\i}a J. \'Alvarez and Rafel Prohens, to whom A.C. Murza is deeply thankful for their kind guidance.


\begin{thebibliography}{00}
\bibitem{Rafel1}{\sc M.J. \'Alvarez, A. Gasull, R. Prohens}, {\it Limit cycles for cubic systems with a symmetry of order 4 and without
infinite critical points }, Proc. Am. Math. Soc. \textbf{136}, (2008), 1035--1043.

\bibitem{Andronov}{\sc A.A. Andronov, E.A. Leontovich, I.I. Gordon, A.G. Maier}, {\it Qualitative theory of second-order dynamic systems},
John Wiley and Sons, New-York (1973).

\bibitem{arn}{\sc V. Arnold}, {\it Chapitres suppl\'ementaires de la th\'eorie des \'equations diff\'erentielles ordinaires},
\'Editions MIR, Moscou, (1978).

\bibitem{carbonel}{\sc M. Carbonell, J. Llibre} {\it Limit cycles of polynomial systems with homogeneous nonlinearities},
J. Math. Anal. Appl. \textbf{142}, (1989), 573--590.

\bibitem{Cherkas}{\sc L.A. Cherkas} {\it On the number of limit cycles of an autonomous second-order system},
Diff. Eq. \textbf{5}, (1976), 666--668.

\bibitem{Che}{\sc S.N Chow, C. Li, D. Wang}, {\it Normal forms and bifurcation of planar vector fields}, Cambridge Univ. Press, (1994).

\bibitem{CGP}\textsc{Coll, B.; Gasull, A.; Prohens, R.} \textit{Differential
equations defined by the sum of two quasi-homogeneous vector
fields.} Can. J. Math., {\bf 49} (1997), pp. 212--231.

\bibitem{Llibre}{\sc A. Gasull, J. Llibre} {\it Limit cycles for a class of Abel Equation}, SIAM J. Math. Anal. \textbf{21}, (1990), 1235--1244.

\bibitem{GS85}{\sc M. Golubitsky, D.G. Schaeffer}, {\it Singularities and groups in bifurcation theory I},
Applied mathematical sciences \textbf{51}, Springer-Verlag, (1985).

\bibitem{GS88}{\sc M. Golubitsky, I. Stewart, D.G. Schaeffer}, {\it Singularities and groups in bifurcation theory II},
Applied mathematical sciences \textbf{69}, Springer-Verlag, (1988).

\bibitem{Lloyd}{\sc N.G. Lloyd} {\it A note on the number of limit cycles in certain two-dimensional systems}, J. London Math. Soc. \textbf{20}, (1979), 277--286.

\bibitem{pliss}{\sc V.A. Pliss}, {\it Non local problems of the theory of oscillations}, Academic Press, New York, (1966).

\bibitem{Zegeling}\textsc{Zegeling, A.} \textit{Equivariant unfoldings in the case
of symmetry of order 4.} Bulgaricae Mathematicae publicationes, {\bf 19} (1993), pp. 71--79.

\end{thebibliography}
\end{document}